\documentclass[12pt,reqno]{amsart}
\usepackage{graphicx}
\usepackage{verbatim}
\usepackage{textcomp}
\usepackage{amssymb}
\usepackage{cite}
\usepackage{amsmath}
\usepackage{latexsym}
\usepackage{amscd}
\usepackage{amsthm}
\usepackage{mathrsfs}
\usepackage{xypic}
\usepackage{bm}
\usepackage{url}
\usepackage{hyperref}
\usepackage{amsfonts}
\usepackage{times,graphicx,hyperref,mathrsfs}
\usepackage{CJK}
\usepackage{color}
\newtheorem{thm}{Theorem}[section]

\newtheorem{lem}[thm]{Lemma}
\newtheorem{prop}[thm]{Proposition}

\theoremstyle{definition}
\newtheorem{defn}{Definition}[section]

\theoremstyle{remark}

\numberwithin{equation}{section}
\def\supp{{\rm{\,supp\,}}}

\allowdisplaybreaks
\begin{document}
\begin{sloppypar}
\title[{Pseudo-differential operators}]
{\uppercase{Weighted weak-type (1, 1) inequalities for pseudo-differential operators with symbol in $S^{m}_{0,\delta}$}}
\author{Guangqing Wang$^{1,3}$}
\author{Suixin He$^{2}$}
\author{Lihua Zhang$^{*3}$}
\address{$1.$ School of Mathematics and Statistics, Fuyang Normal University, Fuyang, Anhui 236041, P.R.China}
\address{$2.$ School of Mathematics and Statistics, Yili Normal University, Yili, Xinjiang 835000, P.R.China}
\address{$3.$School of Mathematics, Sun Yat-sen University, Guangzhou, Guangdong 510275, P.R.China}
\email{wanggqmath@fynu.edu.cn(G.Wang)}
\email{hesuixinmath@126.com(S.He)}
\email{zhanglh89@muil2.sysu.edu.cn(L.Zhang)}
%
\thanks{The research of author is supported in part by Scientific Research Foundation of Education Department of Anhui Province of China (2022AH051320,KJ2021A0659), Doctoral Scientific Research Initiation Project of Fuyang Normal University (2021KYQD0001) and University Excellent Young Talents Research Project of Anhui Province (gxyq2022039).}
\thanks{Corresponding  Author: Lihua Zhang}

\maketitle

\begin{abstract}
Let $T_a$ be a pseudo-differential operator defined by exotic symbol $a$ in H\"{o}rmander class $S^m_{0,\delta}$ with $m \in \mathbb{R} $ and $0 \leq \delta \leq 1 $.
It is well-known that the weak type (1,1) behavior of $T_a $ is not fully understood when the index $m $ is equal to the possibly optimal value $-\frac{n}{2} - \frac{n}{2} \delta $ for $0 \leq \delta < 1 $, and that $T_a $ is not of weak type (1,1) when $m = -n $ and $\delta = 1 $.

In this note, we prove that $T_a $ is of weighted weak type (1,1) if $a \in S^{-n}_{0, \delta}$ with $0 \leq \delta < 1 $. Additionally, we show that the dual operator $T_a^* $ is of weighted weak type (1,1) if $a \in L^\infty S^{-n}_0 $. We also identify $m = -n $ as a critical index for these weak type estimates. As applications, we derive weighted weak type (1,1) estimates for certain classes of Fourier integral operators.

\end{abstract}



\section{Introduction and main results}
The study of pseudo-differential operators (PDOs) originated with foundational work by H"{o}rmander \cite{Hormander1} and Nirenberg \cite{Nirenberg}. These operators are defined by the following formula:
\begin{eqnarray}\label{df}
T_{a}u(x)
&=&\int_{\mathbb{R}^n} e^{ i \langle x,\xi\rangle}a(x,\xi)  \hat{u}(\xi)d\xi,
\end{eqnarray}
where $\hat{u}$ denotes the Fourier transform of $u$ and $a(x,\xi)$ is the symbol or amplitude of the operator. A key class of symbols is the H\"{o}rmander class $S^{m}_{\varrho,\delta}$ \cite{Hormander2}. Let $m\in \mathbb{R}$, $0\leq\varrho,\delta\leq1$. A symbol $a(x,\xi)$ belongs to the H\"{o}rmander class $S^{m}_{\varrho,\delta}$, if $a(x,\xi)\in C^{\infty}(\mathbb{R}^{n}\times\mathbb{R}^{n})$, and for any multi-indices $\alpha,\beta,$ the following inequality holds:
\begin{eqnarray*}
|\partial^{\beta}_x\partial^{\alpha}_{\xi}a(x,\xi)|\leq C_{\alpha,\beta}\langle\xi\rangle^{m-\varrho |\alpha|+\delta|\beta|},
\end{eqnarray*}
where $C_{\alpha, \beta}$ is a constant depending on $\alpha$ and $\beta$.

The theory of pseudo-differential operators (PDOs) has long been an important area of study, particularly with regard to their behavior in both weighted and unweighted $L^p$ spaces. The fundamental result on $L^2$ regularity was established by H\"{o}rmander \cite{Hormander3} and Hounie \cite{Hounie2}, who showed that if $a \in S^{-\frac{n}{2} \max\{\delta - \varrho, 0\}}_{\varrho, \delta}$ with $0 \leq \varrho \leq 1$ and $0 \leq \delta < 1$, then $T_a$ is bounded on $L^2$. For endpoint estimates, Stein, in his unpublished lecture notes, showed that $T_a$ is of weak type (1,1) if $a \in S^{-\frac{n}{2}(1-\varrho)}_{\varrho, \delta}$ and either $0 < \delta = \varrho < 1$ or $0 \leq \delta < \varrho \leq 1$. A more general result for a broader range of $\varrho$ and $\delta$ was later established by \'{A}lvarez and Hounie \cite{Hounie}, which can be summarized as follows:
\begin{thm}[\'{A}lvarez and Hounie \cite{Hounie}]
Let $0<\varrho\leq1$, $0\leq\delta<1$ and $m\in \mathbb{R}$. If $a(x,\xi)\in S^{m}_{\varrho,\delta}$, then $T_{a}$ is of weak type (1,1), provided
$$m\leq-\frac{n}{2}(1-\varrho)-\frac{n}{2}\max\{\delta-\varrho,0\}.$$

\end{thm}
In the special case where $\delta=1$ and $\varrho=0$, Guo et.al. \cite{Guo} construct a symbol $a\in S^{-n}_{0,1}$ such that $T_{a}$ is not weak type (1,1). For results related to the $(H^{1},L^{1})$-estimate and $(L^{\infty},BMO)$-estimate, see Coifman et. al. \cite{Coifman} for $\varrho=\delta=0$, Stein\cite{Stein} for $0<\delta=\varrho<1$ or $0 \leq \delta < \varrho \leq 1$, and Wang\cite{W} or \'{A}lvarez et. al. \cite{Hounie} for $0\leq\varrho<\delta<1$.

The remaining open question concerns endpoint estimates: for $a\in S^{m}_{0,\delta}$ with $0\leq\delta<1$ and $m\leq\frac{n}{2}-\frac{n}{2}\delta$,
is $T_{a}$ of weak type (1,1). While we cannot provide a definitive answer, we can give the result in the weighted case, where the bound on $m$ is sharp.

\begin{thm}\label{Th1}
Let $a\in S^{-n}_{0,\delta}$ with $0\leq\delta<1$ and $\omega\in A_{1}$. Then
\begin{equation*}
\|T_{a}u\|_{L^{1,\infty}_{\omega}}\lesssim\|u\|_{L^{1}_{\omega}}.
\end{equation*}
This result is sharp.
\end{thm}

Here and below, $A_p$ denotes the Muckenhoupt class. A nonnegative locally integrable function $\omega$ belongs to $A_p$ if there exists a constant $C > 0$ such that
\begin{center}
$\sup\limits_{Q\subset\mathbb{R}^{n}}\big(\frac{1}{|Q|}\int_{Q}\omega(x)dx\big)
\big(\frac{1}{|Q|}\int_{Q}\omega(x)^{\frac{1}{1-p}}dx\big)^{p-1}\leq C$ for $1<p<\infty,$
\end{center}
and for $p=1$
\begin{center}
$M\omega(x)\leq C\omega(x)$ for almost all $x\in\mathbb{R}^{n}.$
\end{center}
For $p=\infty$, one define $A_{\infty}:=\cup_{p>1}A_{p}$. The smallest constant appearing in these inequalities is the $A_{p}$ constant of $\omega$ denoted by $[\omega]_{p}$. The following norm notations are used:
$$\|u\|^{p}_{L^{p}_{\omega}}=\int_{\mathbb{R}^{n}}|u(x)|^{p}\omega(x)dx~\mathrm{and}~
\|u\|^{p}_{L^{p,\infty}_{\omega}}=\sup_{\lambda>0}\lambda^{p}\omega(x\in\mathbb{R}^{n}:|u(x)|>\lambda).$$

Consider the dual operators $T^{*}_{a}$ defined by the formula
\begin{eqnarray}\label{df*}
T^{*}_{a}u(x)
&=&\int_{\mathbb{R}^n}\int_{\mathbb{R}^n} e^{ i \langle x-y,\xi\rangle}a(y,\xi)d\xi u(y) dy.
\end{eqnarray}
Interestingly, one can obtain similar estimates for the operators $T^{*}_{a}$ even without smoothness in the frequency variable $\xi$ of $a(x,\xi)$, i.e. $a\in L^{\infty}S^{-n}_{0}$.
Here $L^{\infty}S^{m}_{\varrho}$ denotes the rough H\"{o}rmander class \cite{Kenig,Wolfgang}where $a(x,\xi)$ obeys
 \begin{equation*}
   \|\partial^{\alpha}_{\xi}a(\cdot,\xi)\|_{L^{\infty}(\mathbb{R}^n)}\leq C_{\alpha}\langle\xi\rangle^{m-\varrho |\alpha|}.
\end{equation*}
Clearly, the inclusion $S^{m}_{\varrho,\delta}\subset L^{\infty}S^{m}_{\varrho}$ holds for any $m\in \mathbb{R}$, $1\leq\varrho,\delta\leq1$.

\begin{thm}\label{Th2}
Let $a\in L^{\infty}S^{-n}_{0}$ and $\omega\in A_{1}$. Then
\begin{equation*}\label{weighted3}
\|T^{*}_{a}u\|_{L^{1,\infty}_{\omega}}\lesssim\|u\|_{L^{1}_{\omega}}.
\end{equation*}
This result is sharp.
\end{thm}

The study of weighted $L^{p}$-estimates for pseudo-differential operators has been an active area of research, particularly in the 1980s \cite{Chanillo, MiyachiY, Journe, Hounie}, and has seen further improvements by Michalowski et. al. \cite{Michalowski, Michalowski1} in the late, with more recent work by Wang \cite{CW,W} and Park \cite{Park,Park1}. These results can be summarized as follows:
\begin{thm}[Wang \cite{W}]\label{W}
Let $0\leq\varrho\leq1,$ $0\leq\delta<1$, $1\leq r\leq2$ and $a(x,\xi)\in S^{-\frac{n}{r}(1-\varrho)}_{\varrho,\delta}$. Suppose $\omega\in A_{p/r}$ with $r<p<\infty$. Then there is a constant $C$ independent of $a$ and $u$, such that
\begin{equation}\label{weighted}
\|T_{a}u\|_{L^{p}_{\omega}}\leq C\|u\|_{L^{p}_{\omega}}.
\end{equation}
\end{thm}

The case $r=2$, we refer Chanillo et.al \cite{Chanillo} for $0\leq\delta<\varrho<1$, Miyachi et.al \cite{MiyachiY} for $0<\delta=\varrho<1$ and Michalowski et.al\cite{Michalowski1} for $0<\varrho<\delta<1$. In the case $r=1$ and $\varrho\neq0$, we refer  Michalowski et.al\cite{Michalowski}. In the case $1\leq r\leq2$ and $0<\delta\leq\varrho<1$, we refer Chen \cite{Chen} and Park \cite{Park,Park3}. Notably, the case $p=r=1$ has not been fully explored in the existing literature. Our results address this gap for the case $p=r=1$ and $\varrho=0.$

The remainder of the paper is organized as follows: Section 2 introduces the necessary preliminaries and contains proofs of the main results, while Section 3 show Counterexamples and some applications to weighted estimates for certain classes of Fourier integral operators.
\allowdisplaybreaks

\section{The proof of pointwise estimate for the sharp maximal function}
 When $\varrho>0$, PDOs with symbols in $S^{m}_{\varrho,\delta}$ are pseudo-local, meaning their kernels are smooth away from the diagonal.
However, when $\varrho=0$, these operators are no longer pseudo-local. This departure implies that the Calder\'{o}n-Zygmund composition method is no longer directly applicable to $\varrho=0$. Nonetheless, it has been found that $T_{a}$ is bounded on $L^{1}$ with a suitable order. This property and $L^{2}$-estimate for PDOs allow these operators to enjoy local estimates after applying a dyadic partition of unity. Consequently, the local estimate guarantees the convergence of the resulting series.

We are going to get a pointwise estimate for PDOs with respect to Fefferman-Stein sharp maximal
function and Hardy-Littlewood maximal function.
\begin{prop}\label{P1}
If $a\in S^{-n}_{0,\delta}$ with $0\leq\delta<1$, then
\begin{equation*}\label{MM}
M^{\sharp}(T_{a}f)(x)\lesssim Mf(x).
\end{equation*}
\end{prop}

\begin{prop}\label{P2}
If $a\in L^{\infty}S^{-n}_{0}$ then
$$M^{\sharp}(T^{*}_{a}f)(x)\lesssim Mf(x).$$
\end{prop}
Then, Theorem \ref{Th1} and Theorem \ref{Th2} follows from the famous Fefferman-Stein's inequalities \cite{Stein1}:
$$\|Mu\|_{L^{p}_{\omega}}\lesssim \|M^{\sharp}u\|_{L^{p}_{\omega}},\quad \|Mu\|_{L^{p,\infty}_{\omega}}\lesssim \|M^{\sharp}u\|_{L^{p,\infty}_{\omega}}$$
for $0<p<\infty$ and $\omega\in A_{\infty}$.
Moreover, one can get the following results:
\begin{thm}
Let $a\in S^{-n}_{0,\delta}$ with $0\leq\delta<1$ and $\omega\in A_{p}$ with $1<p<\infty$. Then
\begin{center}
$\|T_{a}u\|_{L^{p}_{\omega}}\lesssim\|u\|_{L^{p}_{\omega}}.$
\end{center}
\end{thm}

\begin{thm}
Let $a\in L^{\infty}S^{-n}_{0}$ and $\omega\in A_{p}$ with $1<p<\infty$. Then
\begin{center}
$\|T^{*}_{a}u\|_{L^{p}_{\omega}}\lesssim\|u\|_{L^{p}_{\omega}}.$
\end{center}
\end{thm}

For a function $u\in L^{1}_{loc}(\mathbb{R}^{n})$, we define the Fefferman-Stein sharp maximal
function and Hardy-Littlewood maximal function by the formula:
$$M^{\sharp}u(x)=\sup\limits_{x\in Q}\inf\limits_{c}\frac{1}{|Q|}\int_{Q}|u(y)-c|dy\quad {\rm and}\quad Mu(x)=\sup\limits_{x\in Q}\frac{1}{|Q|}\int_{Q}|u(y)|dy$$
respectively, where $c$ moves over all complex number, and $Q$ containing $x$ moves over all cubes with its sides parallel to the coordinate axes.

Now we are going to prove Proposition \ref{P1} and Proposition \ref{P2}. Before this, we state two useful facts
\begin{prop}
Let $0\leq\varrho\leq1,$ $0\leq\delta<1$ and $a(x,\xi)\in S^{m}_{\varrho,\delta}$.
If $$m< -\frac{n}{2}(1-\varrho)-\frac{n}{2}\max\{\delta-\varrho,0\},$$
then
$$\|T_{a}u\|_{L^{1}}\lesssim\|u\|_{L^{1}}.$$
\end{prop}

\begin{prop}
Let $0\leq\varrho\leq1$ and $a(x,\xi)\in L^{\infty}S^{m}_{\varrho}$.
If $$m< -\frac{n}{2}(1-\varrho),$$
then
$$\|T^{*}_{a}u\|_{L^{1}}\lesssim\|u\|_{L^{1}}.$$
\end{prop}
Their proofs is trivial. See \cite{Hounie} for the case $0<\varrho\leq1$, the proofs are contained in $(H^{1},L^{1})$ estimates for $T_{a}$ and $T^{*}_{a}$ respectively. The case $\varrho=0$ can be gotten by a similar argument.

Let \begin{eqnarray}\label{K}
K(x,y)
=\frac{1}{(2\pi)^n}\int_{\mathbb{R}^n} e^{ i \langle x-y,\xi\rangle}a(x,\xi)d\xi
~{\rm and}~
K^{*}(x,y)
=\frac{1}{(2\pi)^n}\int_{\mathbb{R}^n} e^{ i \langle x-y,\xi\rangle}a(y,\xi)d\xi.
\end{eqnarray}
Then $T_{a}$ and $T^{*}_{a}$ can be written as
\begin{eqnarray}
T_{a}u(x)
=\int_{\mathbb{R}^n}K(x,y) u(y)dy \quad
{\rm and}\quad
T^{*}_{a}u(x)
=\int_{\mathbb{R}^n}K^{*}(x,y) u(y)dy
\end{eqnarray}
respectively. Now we introduce the standard Littlewood-Paley partition of unity. Let $C>1$ be a constant. Set
$E_{-1}=\{\xi:|\xi|\leq 2C\}$, $E_{j}=\{\xi:C^{-1}2^{j}\geq|\xi|\leq C2^{j+1}\}$, $j=0,1,2,\cdots$.

\begin{lem}\label{L0}
There exist $\psi_{-1}(\xi),\psi(\xi)\in C^{\infty}_{0}$, such that
\begin{enumerate}
  \item $\supp\psi\subset E_{0}$, $\supp\psi_{-1}\subset E_{-1};$
  \item $0\leq\psi\leq1$, $0\leq\psi_{-1}\leq1;$
  \item $\psi_{-1}(\xi)+\sum\limits^{\infty}_{j=1}\psi(2^{-j}\xi)=1.$
\end{enumerate}
\end{lem}
By Lemma \ref{L0}, the symbol $a(x,\xi)$ can been written as
$$a(x,\xi)=a(x,\xi)\big(\psi_{-1}(\xi)+\sum\limits^{\infty}_{j=1}\psi(2^{-j}\xi)\big)=:\sum\limits^{\infty}_{j=0}a_{j}(x,\xi).$$
Consequently, the operator $T_{a}$ and $T^{*}_{a}$ can been decomposed as
\begin{eqnarray}\label{de}
T_{a}u(x)=\sum\limits_{j=0}^{\infty}T_{j}u(x)
~
 \mathrm{and}~
T^{*}_{a}u(x)=\sum\limits_{j=0}^{\infty}T^{*}_{j}u(x),
\end{eqnarray}
respectively, where
\begin{eqnarray}
T_{j}u(x)
=\int_{\mathbb{R}^n}K_{j}(x,y) u(y)dy \quad
{\rm with}\quad
K_{j}(x,y)
=\frac{1}{(2\pi)^n}\int_{\mathbb{R}^n} e^{ i \langle x-y,\xi\rangle}a_{j}(x,\xi)d\xi
\end{eqnarray}
\begin{eqnarray}
T^{*}_{j}u(x)
=\int_{\mathbb{R}^n}K^{*}_{j}(x,y) u(y)dy
\quad
{\rm with}\quad
K^{*}_{j}(x,y)
=\frac{1}{(2\pi)^n}\int_{\mathbb{R}^n} e^{ i \langle x-y,\xi\rangle}a_{j}(y,\xi)d\xi
\end{eqnarray}

In this note, we denote by $Q(x_{0},l)$ a fixed cube centered at $x_{0}$ with the side length $l>0$.
\begin{lem}\label{L1}
Let $l<1$ and $j$ be a positive integer. If $a\in S^{-n}_{0,\delta}$ with $0\leq\delta<1$, then
\begin{eqnarray}\label{E1}
\int_{\mathbb{R}^{n}}|u(y)||K_{j}(x,y)-K_{j}(z,y)|dy\lesssim 2^{j}lMu(x_{0}),\quad \forall x,z\in Q(x_{0},l)
\end{eqnarray}
If $a\in L^{\infty}S^{-n}_{0}$, then
\begin{eqnarray}\label{E2}
\int_{\mathbb{R}^{n}}|u(y)||K^{*}_{j}(x,y)-K^{*}_{j}(z,y)|dy\lesssim 2^{j}lMu(x_{0}),\quad \forall x,z\in Q(x_{0},l).
\end{eqnarray}
\end{lem}

\begin{proof}
The proof of (\ref{E1})  and (\ref{E2}) is standard. We only show a outline of proving (\ref{E2}).
Integrand of left side of (\ref{E2}) can be bounded by
\begin{eqnarray}\label{E40}
\int_{\mathbb{R}^{n}}|u(y)||\int_{\mathbb{R}^n}\big(e^{i\langle x-y,\xi\rangle}-e^{ i \langle z-y,\xi\rangle}\big)a_{j}(y,\xi)d\xi|dy.
\end{eqnarray}
Break up this integrand as follows
\begin{eqnarray*}
\int_{|y-x_{0}|\leq2}
+
\int_{|y-x_{0}|>2}.
\end{eqnarray*}
A direct calculation gives the first term is bounded by $2^{j}lMu(x_{0})$, and integration by parts with respect to the variable $\xi$ yields that the second term has the same bound.
Thus the proof is completed.
\end{proof}

\begin{lem}\label{L2}
Let $l<1$ and $j$ be a positive integer satisfying $2^{j}\geq l^{-1}$. If
$a\in S^{-n-\frac{n}{2}\delta}_{0,\delta}$ with $0\leq\delta<1$
then
\begin{eqnarray}\label{E3}
\frac{1}{|Q(x_{0},l)|}\int_{Q(x_{0},l)}|T_{j}u(x)|dx
&\lesssim&2^{-j\frac{n}{2}}l^{-\frac{n}{2}}Mu(x_{0}).
\end{eqnarray}
If $a\in L^{\infty}S^{-n}_{0}$, then
\begin{eqnarray}\label{E4}
\frac{1}{|Q(x_{0},l)|}\int_{Q(x_{0},l)}|T^{*}_{j}u(x)|dx
&\lesssim&2^{-j\frac{n}{2}}l^{-\frac{n}{2}}Mu(x_{0}).
\end{eqnarray}
\end{lem}

\begin{proof}
We prove (\ref{E3}) first. H\"{o}lder's inequality and Minkowski's inequality implies that the left hand in (\ref{E3}) can be bounded by
\begin{eqnarray*}
l^{-\frac{n}{2}}\int_{\mathbb{R}^{n}}\big(\int_{|x-x_{0}|<l}|K_{j}(x,y)|^{2}dx\big)^{\frac{1}{2}}|u(y)|dy.
\end{eqnarray*}
So, it suffices to show
\begin{eqnarray*}
\int_{\mathbb{R}^{n}}\big(\int_{|x-x_{0}|<l}|K_{j}(x,y)|^{2}dx\big)^{\frac{1}{2}}|u(y)|dy
&\lesssim&2^{-j\frac{n}{2}}Mu(x_{0}).
\end{eqnarray*}
Break up the integral with respect to the variable $y$ as follows
\begin{eqnarray}\label{E01}
\int_{|y-x_{0}|\leq2}+\int_{|y-x_{0}|>2}.
\end{eqnarray}
Let $c_{j}(x,\xi)=a_j(x,\xi)|\xi|^{n}$ and $\widehat{h_{j}(\xi)}=|\xi|^{-n}\chi_{j}(\xi)$. Then we can write
\begin{eqnarray*}
K_{j}(x,y)
=T_{c_{j}}h_{j}(x-y)
\end{eqnarray*}
So, the first term in (\ref{E01}) can be wrote as
\begin{eqnarray*}
\int_{|y-x_{0}|\leq2}\big(\int_{\mathbb{R}^{n}}|T_{c_{j}}h_{j}(x-y)|^{2}dx\big)^{\frac{1}{2}}|u(y)|dy
\end{eqnarray*}
Notice $c_{j}\in S^{-\frac{n}{2}\delta}_{0,\delta}$. Moreover $T_{c_{j}}$ is bounded on $L^{2}$. So it can be bounded by
\begin{eqnarray*}
\int_{|y-x_{0}|\leq2}|u(y)|dy\big(\int_{\mathbb{R}^{n}}|h_{j}(\xi)|^{2}d\xi\big)^{\frac{1}{2}}
\leq 2^{-j\frac{n}{2}}Mu(x_{0}).
\end{eqnarray*}

Now we estimate the second term in (\ref{E01}). For positive integer $N>n/2$, denote $\tilde{c}_{j}(x,\xi)=\Delta^{N}_{\xi}a_j(x,\xi)|\xi|^{n}$ and $\widehat{\tilde{h}_{j}(\xi)}=|\xi|^{-n}\chi_{j}(\xi)$. Then we can write
\begin{eqnarray*}
K_{j}(x,y)
=\frac{1}{|x-y|^{2N}}\int_{{\mathbb{R}^n}}e^{ i\langle x-y, \xi\rangle}\Delta^{N}_{\xi}a_j(x,\xi)d\xi
=\frac{1}{|x-y|^{2N}}T_{\tilde{c}_{j}}\tilde{h}_{j}(x-y)
\end{eqnarray*}
So, the second term in (\ref{E01}) can be wrote as
\begin{eqnarray*}
\int_{|y-x_{0}|>2}\big(\int_{|x-x_{0}|<l}
|\frac{1}{|y-x|^{2N}}T_{\tilde{c}_{j}}\tilde{h}_{j}(x-y)|^{2}dx\big)^{\frac{1}{2}}|u(y)|dy.
\end{eqnarray*}
Notice that $|y-x|\sim|y-x_{0}|$ for any $|x-x_{0}|<l<1$ and $|y-x_{0}|>2\geq2l$. Then it is bounded by
\begin{eqnarray*}
\int_{|y-x_{0}|>2}\frac{1}{|y-x_{0}|^{2N}}\big(\int_{\mathbb{R}^{n}}
|T_{\tilde{c}_{j}}\tilde{h}_{j}(x-y)|^{2}dx\big)^{\frac{1}{2}}|u(y)|dy.
\end{eqnarray*}
Clearly, $\tilde{c_{j}}\in S^{-\frac{n}{2}\delta}_{0,\delta}$. So $L^{2}$ boundedness of $T_{\tilde{c_{j}}}$ gives that it has bound
\begin{eqnarray*}
\int_{|y-x_{0}|>2}\frac{1}{|y-x_{0}|^{2N}}|u(y)|dy\big(\int_{\mathbb{R}^{n}}|\tilde{h}_{j}(\xi)|^{2}d\xi\big)^{\frac{1}{2}}
\leq 2^{-j\frac{n}{2}}Mu(x_{0}).
\end{eqnarray*}

For (\ref{E4}), it be got by the same argument as above with $L^{2}$ boundedness of pseudo-differential operators replaced by Parseval's identity and there is no use for the smoothness of variable $y$ of $a(y,\xi)$. So the proof is finished.
\end{proof}

\begin{lem}\label{L3}
Let $l<1$ and $j$ be a positive integer satisfying $2^{j}\geq l^{-1}$. If $a\in S^{-n}_{0,\delta}$ with $0<\delta<1$, then for any $\lambda\geq1$,
\begin{eqnarray*}
\frac{1}{|Q|}\int_{Q(x_{0},l)}|T_{j}u(x)|dx
&\lesssim&\big(2^{j\delta}l^{\lambda}+l^{-\frac{n\lambda}{2}}2^{-j\frac{n}{2}}\big)Mu(x_{0}).
\end{eqnarray*}
\end{lem}

\begin{proof}
If $\lambda>1$, then $l^{\lambda}<l$. Take integer $L$ such that it is the first number no less than  $l^{1-\lambda}$, that is $L-1<l^{1-\lambda}\leq L$. Then there are $L^{n}$ cubes with the same side length $l^{\lambda}$ covering $Q(x_{0},l)$. Moreover, we have
$$Q(x_{0},l)\subset\cup_{i=1}^{L^{n}}Q(x_{i},l^{\lambda})\subset Q(x_{0},2l).$$
Clearly, $L^{n}\leq2^{n}l^{n(1-\lambda)}$. Denote
\begin{eqnarray}\label{0A}
T_{j,i}u(x)
&=&\int_{\mathbb{R}^n} e^{ i \langle x,\xi\rangle}a(x_{i},\xi) \psi(2^{-j}\xi) \hat{u}(\xi)d\xi.
\end{eqnarray}
\begin{eqnarray*}
&&\frac{1}{|Q|}\int_{Q(x_{0},l)}|T_{j}u(x)|dx\nonumber\\
&\leq&\frac{1}{|Q|}\sum\limits_{i=1}^{L^{n}}\bigg(\int_{Q(x_{i},l^{\lambda})}|T_{j}u(x)-T_{j,i}u(x)|dx+\int_{Q(x_{i},l^{\lambda})}|T_{j,i}u(x)|dx\bigg).
\end{eqnarray*}
Lemma \ref{L1} and Lemma \ref{L2} imply
\begin{eqnarray*}
|T_{j}f( x)-T_{j,i}f( x)|\lesssim l^{\lambda}2^{j\delta}Mu(x_{0})
\end{eqnarray*}
and
\begin{eqnarray*}
\int_{Q(x_{i},l^{\lambda})}|T_{j,i}u(x)|dx
&\lesssim&2^{-j\frac{n}{2}}l^{\frac{n\lambda}{2}}Mu(x_{0}),
\end{eqnarray*}
respectively. Recall $L^{n}\leq2^{n}l^{n(1-\lambda)}$, the desired estimate can be gotten immediately.

If $\lambda=1$, we define
\begin{eqnarray}
T_{j,0}u(x)
&=&\int_{\mathbb{R}^n} e^{ i \langle x,\xi\rangle}a(x_{0},\xi) \psi(2^{-j}\xi) \hat{u}(\xi)d\xi.
\end{eqnarray}
Then the desired estimate can be got by the same argument as above with $T_{j,i}u$ replaced by $T_{j,0}u$. So we complete the proof.
\end{proof}

Here, the range of $\lambda$ can be extended to $[1,\infty)$.
However, to make some sums convergent, $\lambda$ has to be confined to a finite range. To this end, fix positive integer $N_{\delta}>\frac{2n}{1-\delta}+n$ and take $$\max\{0,\frac{n}{2}(1-\delta)-\frac{nN_{\delta}}{N_{\delta}-n}\}<\theta_{0}<\frac{n}{2}(1-\delta)-\frac{n^{2}}{N_{\delta}-n}.$$
Denote
\begin{equation}\label{E5}
T_{\delta}=-\frac{n^{2}}{N_{\delta}}+(1-\frac{n}{N_{\delta}})(\frac{n}{2}(1-\delta)-\theta_{0}).
\end{equation}
Clearly $\theta_{0}>0$ is well defined and $T_{\delta}>0.$

\begin{lem}\label{L4}
Let $Q(x_{0},l)$ be a fixed cube with side length $l<1$. For any positive integer $j$ with $l^{-\frac{n-n^{2}/N_{\delta}}{T_{\delta}}}\leq2^{j}$, if $a\in S^{-n}_{0,\delta}$ with $0\leq\delta<1$ then
\begin{eqnarray*}
\frac{1}{|Q|}\int_{Q(x_{0},l)}|T_{j}u(x)|dx
&\lesssim&2^{-jT_{\delta}}l^{\frac{n^{2}}{N_{\delta}}-n}Mu(x_{0}),
\end{eqnarray*}
where $T_{\delta}$ is given by (\ref{E5}).
\end{lem}

\begin{proof}
Denote
$$\Gamma_{0}=l^{\frac{n}{N_{\delta}}}2^{j\frac{n}{N_{\delta}}+j\frac{1}{N_{\delta}}\big(\frac{n}{2}(1-\delta)-\theta_{0}\big)},$$
Set $u_{1}(x)=u(x)\chi_{Q(x_{0},2\Gamma_{0})}(x)$ and $u_{2}(x)=u(x)-u_{1}(x)$.
Then
\begin{eqnarray}\label{E19}
\frac{1}{|Q|}\int_{Q(x_{0},l)}|T_{j}u(x)|dx
\leq\frac{1}{|Q|}\int_{Q(x_{0},l)}|T_{j}u_{1}(x)|dx+
\frac{1}{|Q|}\int_{Q(x_{0},l)}|T_{j}u_{2}(x)|dx.
\end{eqnarray}

Notice that $a_{j}(x,\xi)\in S^{-\frac{n}{2}-\frac{n}{2}\delta-\theta_{0}}_{0,\delta}$ with bounds $\lesssim 2^{-j\frac{n}{2}(1-\delta)+j\theta_{0}}$. The $L^{1}$-estimate of $T_{j}$ (Proposition \ref{P1}) gives that
\begin{eqnarray}\label{E20}
\frac{1}{|Q|}\int_{Q}|T_{j}u_{1}(x)|dx
&\lesssim& 2^{-j\frac{n}{2}(1-\varrho)+j\theta_{0}}\Gamma_{0}^{n}l^{-n}Mu(x_{0})=2^{-jT_{\delta}}l^{\frac{n^{2}}{N_{\delta}}-n}Mu(x_{0}).
\end{eqnarray}

Recall $\max\{0,\frac{n}{2}(1-\delta)-\frac{nN_{\delta}}{N_{\delta}-n}\}<\theta_{0}<\frac{n}{2}(1-\delta)-\frac{n^{2}}{N_{\delta}-n}$ and  $2^{j}\geq l^{-\frac{n-n^{2}/N_{\delta}}{T_{\delta}}}$. Then we have $\frac{n-n^{2}/N_{\delta}}{T_{\delta}}>1$ and  $2^{j}\geq l^{-\frac{n-n^{2}/N_{\delta}}{T_{\delta}}}\geq l^{-1}$ which implies that
$$\Gamma_{0}=(l2^{j})^{\frac{n}{N_{\delta}}}2^{j\frac{1}{N_{\delta}}\big(\frac{n}{2}(1-\delta)-\theta_{0}\big)}>1>l.$$
Consequently, one have
$|y-x|\sim|y-x_{0}|$ for $\forall x\in Q(x_{0},l)$ and $\forall y\in Q^{C}(x_{0},2\Gamma_{0})$. Integrating by parts $N_{\delta}$-th with respect to $\xi$ gives that
\begin{eqnarray}\label{E21}
|T_{j}u_{2}(x)|
\lesssim2^{jn}\Gamma_{0}^{(n-N_{\delta})}
Mu(x_{0})=2^{-jT_{\delta}}l^{\frac{n^{2}}{N_{\delta}}-n}Mu(x_{0}).
\end{eqnarray}
Clearly, the desired estimate follows from \eqref{E19}, \eqref{E20} and \eqref{E21}.
\end{proof}

Now we consider the case $l\geq1$, which can be done by a similar method as above lemma.

\begin{lem}\label{L5}
Let $l\geq1$ and $j$ be a positive integer. Suppose $0\leq\delta<1$, $0<\theta_{1}<\frac{n}{2}(1-\delta)$ and $\epsilon>0$.
If $a\in S^{-n}_{0,\delta}$, then for any positive integer $N>n$
\begin{equation}\label{E7}
\frac{1}{|Q|}\int_{Q(x_{0},l)}|T_{j}u(x)|dx
\lesssim(2^{-j\frac{n}{2}(1-\delta)+j\theta_{1}+j\epsilon n}+2^{-j\epsilon(N-n)}l^{-(N-n)})Mu(x_{0}).
\end{equation}
If $a\in L^{\infty}S^{-n}_{0}$, then for any positive integer $N>n$
\begin{equation}\label{E8}
\frac{1}{|Q|}\int_{Q(x_{0},l)}|T^{*}_{j}u(x)|dx
\lesssim(2^{-j\frac{n}{2}+j\theta_{1}+j\epsilon n}+2^{-j\epsilon(N-n)}l^{-(N-n)})Mu(x_{0}).
\end{equation}

\end{lem}

\begin{proof}
Denote
$$\Gamma=l2^{j\epsilon},$$
Set $u_{3}(x)=u(x)\chi_{Q(x_{0},2\Gamma)}(x)$ and $u_{4}(x)=u(x)-u_{3}(x)$.
Then
\begin{eqnarray}
\frac{1}{|Q|}\int_{Q(x_{0},l)}|T_{j}u(x)|dx
\leq\frac{1}{|Q|}\int_{Q(x_{0},l)}|T_{j}u_{3}(x)|dx+
\frac{1}{|Q|}\int_{Q(x_{0},l)}|T_{j}u_{4}(x)|dx.
\end{eqnarray}

Notice that $a_{j}(x,\xi)\in S^{-\frac{n}{2}-\frac{n}{2}\delta-\theta_{1}}_{0,\delta}$ with bounds $\lesssim 2^{-j\frac{n}{2}(1-\delta)+j\theta_{1}}$. The $L^{1}$-estimate of $T_{j}$ give that
\begin{eqnarray}
\frac{1}{|Q|}\int_{Q}|T_{j}u_{3}(x)|dx
&\lesssim& 2^{-j\frac{n}{2}(1-\varrho)+j\theta_{1}+j\epsilon n}Mu(x_{0}).
\end{eqnarray}

Notice that $|y-x|\sim|y-x_{0}|$ for $\forall x\in Q(x_{0},l)$ and $\forall y\in Q^{C}(x_{0},2\Gamma)$. So integrating by parts gives that
\begin{eqnarray}
|T_{j}u_{4}(x)|
\lesssim2^{-j\epsilon(N-n)}l^{-(N-n)}
Mu(x_{0}).
\end{eqnarray}
Clearly, the desired estimate follows from \eqref{E19}, \eqref{E20} and \eqref{E21}.
\end{proof}

\begin{proof}[Proof of Theorem \ref{Th2}]
Without loss of generality, we assume that the symbol $a(x,\xi)$ vanishes for $|\xi|\leq 1$. Let $Q=Q(x_{0},l)$ denote the cube centered at $x_{0}$ with the side length $l.$
We are going to prove
\begin{eqnarray}\label{E00}
\frac{1}{|Q|}\int_{Q}|T^{*}_{a}u(x)-C_{Q}|dx\lesssim Mu(x_{0}),
\end{eqnarray}
where $C_{Q}=\frac{1}{|Q|}\int_{Q}T^{*}_{a}u(y)dy.$ By Lemma \ref{L0}, it suffices to prove
\begin{eqnarray}\label{E9}
\sum_{j}\frac{1}{|Q|^{2}}\int_{Q}\int_{Q}|T^{*}_{j}u(x)-T^{*}_{j}u(y)|dydx\lesssim Mu(x_{0}).
\end{eqnarray}

For the case $l\geq1,$ the left hand above can be controlled by
$$\sum_{j}\frac{2}{|Q|}\int_{Q(x_{0},l)}|T^{*}_{j}u(x)|dx.$$
Applying Lemma \ref{L5} with taking $0<\epsilon<\frac{1}{2}-\frac{\theta_{1}}{n}$, one can get the desired estimate immediately.

For the case $l<1$. Break up the sum in the left hand in (\ref{E9}) as follows
\begin{eqnarray*}
\sum\limits_{2^{j}<l^{-1}}\frac{1}{|Q|^{2}}\int_{Q}\int_{Q}|T^{*}_{j}u(x)-T^{*}_{j}u(y)|dydx
+\sum\limits_{l^{-1}<2^{j}}\frac{2}{|Q|}\int_{Q(x_{0},l)}|T^{*}_{j}u(x)|dx
\end{eqnarray*}
Then, (\ref{E2}) and (\ref{E4}) give that the corresponding sum is bounded by
$$ \sum\limits_{2^{j}<l^{-1}}2^{j}lMu(x_{0})+\sum\limits_{l^{-1}<2^{j}}2^{-j\frac{n}{2}}l^{-\frac{n}{2}}Mu(x_{0})
\lesssim Mu(x_{0}),$$
respectively. So the proof is finished.
\end{proof}

\begin{proof}[Proof of Theorem \ref{Th1}]

It can be proved by a similar argument as above. Clearly, it suffices to show
\begin{eqnarray}\label{E6}
\sum_{j}\frac{1}{|Q|^{2}}\int_{Q}\int_{Q}|T_{j}u(x)-T_{j}u(y)|dydx.
\end{eqnarray}

For the case $l\geq1$. By (\ref{E7}) in Lemma \ref{L5}, we can get that it can bounded by
$$\sum_{j}(2^{-j\frac{n}{2}(1-\delta)+j\theta_{1}+j\epsilon n}+2^{-j\epsilon(N-n)}l^{-(N-n)})Mu(x_{0}).$$
Take $0<\epsilon<\frac{1-\delta}{2}-\frac{\theta_{1}}{n}$, then the sum is convergent.

For the case $l<1$. If $\delta=0$, break up the sum in (\ref{E6}) as follows
$$\sum\limits_{2^{j}<l^{-1}}
+\sum\limits_{l^{-1}<2^{j}}.$$
Then Lemma \ref{L1} and Lemma \ref{L2} imply that (\ref{E6}) can be bounded by
$$\sum\limits_{2^{j}<l^{-1}}2^{j}l Mu(x_{0})
+\sum\limits_{l^{-1}<2^{j}}2^{-j\frac{n}{2}}l^{-\frac{n}{2}}Mu(x_{0})\lesssim Mu(x_{0}).$$
If $0<\delta<1$, break up the sum further
$$\sum\limits_{2^{j}<l^{-1}}+\sum\limits_{l^{-1}\leq2^{j}\leq l^{-(n-n^{2}/N_{\delta})/T_{\delta}}}
+\sum\limits_{l^{-(n-n^{2}/N_{\delta})/T_{\delta}}<2^{j}},$$
where $N_{\delta}$ and $T_{\delta}$ are given in Lemma \ref{L4}. Using Lemma \ref{L1} and Lemma \ref{L4}, one can get that the first term and the last term are bounded by
$$\sum\limits_{2^{j}<l^{-1}}2^{j}l Mu(x_{0})
+\sum\limits_{l^{-(n-n^{2}/N_{\delta})/T_{\delta}}<2^{j}}2^{-jT_{\delta}}l^{\frac{n^{2}}{N_{\delta}}-n}Mu(x_{0})\lesssim Mu(x_{0}).$$
For the second term, we write
\begin{eqnarray*}
\sum\limits_{l^{-1}<2^{j}\leq l^{-(n-n^{2}/N_{\delta})/T_{\delta}}}\frac{2}{|Q|}\int_{Q}|T_{j}u(x)|dx
&=&\big(\sum\limits_{l^{-1}<2^{j}\leq l^{-\frac{1}{\delta}}}
+\sum\limits_{l^{-\frac{1}{\delta}}<2^{j}\leq l^{-\frac{1}{\delta^{2}}}}+...
+\sum\limits_{l^{-\frac{1}{\delta^{k-1}}}<2^{j}\leq l^{-\frac{1}{\delta^{k}}}}\\
&+&...
+\sum\limits_{l^{-\frac{1}{\delta^{\gamma-1}}}<2^{j}\leq l^{-(n-n^{2}/N_{\delta})/T_{\delta}}}\big)
\frac{2}{|Q|}\int_{Q(x_{0},l)}|T_{j}u(x)|dx,
\end{eqnarray*}
where $\gamma$ is the first positive integer such that $\frac{1}{\delta^{\gamma}}\geq \frac{n-n^{2}/N_{\delta}}{T_{\delta}}$.
Then take $\lambda=\frac{1}{\delta^{k}}$, $k=0,1,...,\gamma-1$ in Lemma \ref{L3} respectively, we can see that each sum above is bounded by $Mu(x_{0})$.
\end{proof}

\section{Applications to Fourier Integral Operators}

Fourier integral operators (FIOs), introduced by H\"{o}rmander in \cite{HormanderI}, can be formally expressed as
\begin{eqnarray}\label{F1}
T_{a,\varphi}u(x)
&=& \int_{\mathbb{R}^n} e^{i \varphi(x,\xi)} a(x,\xi) \hat{u}(\xi) \, d\xi,
\end{eqnarray}
where $a$ is the amplitude function and $\varphi$ is the phase function. The regularity of these operators has been extensively studied since the 1980s in works such as \cite{HormanderI,Tao,SSS,Eskin,Beals,Littman,Miyachi,Peral,Beals2,Greenleaf}. More recently, some results have been revisited in \cite{Wolfgang,Wolfgang1,SZ,WCY,WCY1,LW}. However, to the best of our knowledge, there has been limited research on the weighted regularity of these operators, with notable contributions in \cite{Ruzhansky,Wolfgang}.

In this study, we focus on a class of phase functions, originally defined by Dos Santos Ferreira et.al. \cite{Wolfgang}.
\begin{defn}
\textit{\((\Phi^k)\)} A real-valued function \( \varphi(x,\xi) \) belongs to the class \( \Phi^k \) if the following conditions hold:
\begin{itemize}
    \item \( \varphi(x,\xi) \in C^\infty(\mathbb{R}^n \times \mathbb{R}^n \setminus \{0\}) \),
    \item \( \varphi(x,\xi) \) is positively homogeneous of degree 1 in the frequency variable \( \xi \),
    \item For any multi-indices \( \alpha \) and \( \beta \) with \( |\alpha| + |\beta| \geq k \), there exists a positive constant \( C_{\alpha,\beta} \) such that
    \[
    \sup_{(x,\xi) \in \mathbb{R}^n \times \mathbb{R}^n \setminus \{0\}} |\xi|^{-1 + |\alpha|} |\partial_{\xi}^\alpha \partial_{x}^\beta \varphi(x,\xi)| \leq C_{\alpha,\beta}.
    \]
\end{itemize}
\end{defn}

\begin{defn}
\textit{\((L^\infty \Phi^k)\)} A real-valued function \( \varphi(x,\xi) \) belongs to the phase class \( L^\infty \Phi^k \) if it satisfies the following conditions:
\begin{itemize}
    \item \( \varphi(x,\xi) \) is positively homogeneous of degree 1 in the frequency variable \( \xi \),
    \item \( \varphi(x,\xi) \) is smooth on \( \mathbb{R}^n\backslash\{0\} \) in the frequency variable \( \xi \) and is bounded measurable in the spatial variable \( x \),
    \item For all multi-indices \( |\alpha| \geq k \), it satisfies
    \[
    \sup_{\xi \in \mathbb{R}^n_*} |\xi|^{-1+|\alpha|} \|\partial_{\xi}^{\alpha} \varphi(x,\xi)\|_{L^\infty(\mathbb{R}^n)} \leq C_{\alpha},
    \]
    where \( C_{\alpha} \) is a positive constant depending only on \( \alpha \).
\end{itemize}
\end{defn}

Let $ b(x, \xi) = a(x, \xi) e^{i (\varphi(x, \xi) - \langle x, \xi \rangle)} $, so that we can rewrite $ T_{a, \varphi} u(x) $ as:
\[
T_{a, \varphi} u(x) = \int_{\mathbb{R}^n} e^{i \langle x, \xi \rangle} e^{i (\varphi(x, \xi) - \langle x, \xi \rangle)} a(x, \xi) \hat{u}(\xi) \, d\xi = T_b u(x).
\]
If $ a \in S^{-n}_{1,\delta} $ and $ \varphi(x,\xi) - \langle x,\xi \rangle \in \Phi^1 $, then $ b(x, \xi) \in S^{-n}_{0,\delta} $. Therefore, by Theorem \ref{Th1}, we obtain the following result:

\begin{thm}\label{Th3}
Let $ a \in S^{-n}_{1,\delta} $ with $ 0 \leq \delta < 1 $ and $ \varphi(x, \xi) - \langle x, \xi \rangle \in \Phi^1 $. Then the operators $ T_{a, \varphi} $ defined by (\ref{F1}) satisfy:
\[
\| T_{a, \varphi} u \|_{L^{1, \infty}_{\omega}} \lesssim\| u \|_{L^1_{\omega}}.
\]
\end{thm}

Similarly, defining the dual operator by

\begin{eqnarray}\label{F2}
T^{*}_{a,\varphi}u(x)
&=& \int_{\mathbb{R}^n} \int_{\mathbb{R}^n} e^{i(\varphi(y,\xi) - \langle x,\xi\rangle)} a(y,\xi) \, d\xi \, u(y) \, dy,
\end{eqnarray}
we obtain the following result:

\begin{thm}\label{Th4}
Let $ a \in L^\infty S^{-n}_{1} $ and $ \varphi(x, \xi) - \langle x, \xi \rangle \in L^\infty \Phi^1 $. Then the operators $ T^{*}_{a, \varphi} $ defined by (\ref{F2}) satisfy:
\[
\| T^{*}_{a, \varphi} u \|_{L^{1, \infty}_{\omega}} \lesssim \| u \|_{L^1_{\omega}}.
\]
\end{thm}

Under the condition of Theorem \ref{Th3}, we can use Proposition \ref{P1} to obtain:
\[
M^{\sharp}(T_{a,\varphi} u)(x) \lesssim M u(x).
\]
Moreover, $ T_{a,\varphi} $ is a bounded operator on $ L^{p}_{\omega} $ for $ 1 < p < \infty $ and $ \omega \in A_p $. A similar result holds for $ T^{*}_{a,\varphi} $ under the condition of Theorem \ref{Th4}.

$\bf{Counterexamples:}$
The sharpness of Theorem \ref{Th1} and Theorem \ref{Th2} follows from following counterexample going back to \cite{Kurtz}.
Let $0>m>-n$ and $a(x,\xi)=e^{ i\xi_{1}}(1+|\xi|^{2})^{\frac{m}{2}}$, where $\xi_{1}$ is the first component of $\xi\in\mathbb{R}^{n}$. Clearly $a(x,\xi)\in S^{m}_{0,0}$. Consider PDOs
$$T_{a}u(x)=\int_{\mathbb{R}^n} e^{ i \langle x,\xi\rangle+i\xi_{1}} (1+|\xi|^{2})^{\frac{m}{2}}\hat{u}(\xi)d\xi$$
\bibliographystyle{Plain}

\begin{thebibliography}{10}
\bibitem{Hounie} J. \'{A}lvarez, J. Hounie, Estimates for the kernel and continuity properties of pseudo-differential operators. Arkiv f\"{o}r matematik, 1990, 28(1): 1-22.











\bibitem{Beals2}  R. Beals, Spatially inhomogeneous pseudodifferential operators, II. Communications on Pure and Applied Mathematics, 1974, 27(2): 161-205.

\bibitem{Beals}  M. Beals,  $L^p$ boundedness of Fourier integral operators,  Mem.  Amer. Math.  Soc.  264 (1982).

\bibitem{Coifman} R. Coifman, Y.Meyer,  Au del\`{a} des op\'{e}rateurs pseudo-diff\'{e}rentiels. Ast\'{e}risque, 1978, 57.
\bibitem{Wolfgang1} A. Castro, A. Israelsson, W. Staubach, Regularity of Fourier integral operators with amplitudes in general H\"{o}rmander classes. Analysis and Mathematical Physics, 2021, 11(3):121.



\bibitem{Chen} W.Y. Chen, Weighted modulus estimates for pseudodifferential operators and the conjecture of Chanillo-Torchinsky, Chinese Ann. Math. Ser. A 12 (1991) 6-11.

\bibitem{Chanillo} S. Chanillo, A. Torchinsky, Sharp function and weighted $L^{p}$ estimates for a class of pseudo-differential operators, Arkiv f\"{o}r matematik, 1986, 24(1): 1-25.






\bibitem{Wolfgang}  D. Dos Santos Ferreira, W. Staubach, Global and local regularity of Fourier integral operators on weighted and unweighted spaces, American mathematical society, 2014, 229.

\bibitem{Eskin}  G.  I. \`{ E}skin,  Degenerate elliptic pseudo-differential operators of principal type (Russian),  Mat.  Sbornik 1970, 82 (124): 585-628; English translation,  Math.  USSR Sbornik 1970 11:  539-82.

\bibitem{Stein1} C. Fefferman, E. Stein. $H^p$ spaces of several variables. Acta Mathematica, 1972, 129(3-4): 137-193.
\bibitem{Guo} J. Guo, X. Zhu, Some notes on endpoint estimates for pseudo-differential operators. Mediterranean Journal of Mathematics, 2022, 19(6): 260

\bibitem{Hormander3} L. H\"{o}rmander, On the $L^{2}$ continuity of pseudo-differential operators. Communications on Pure and Applied Mathematics, 1971, 24(4): 529-535.


\bibitem{Hounie2} J. Hounie, On the $L^{2}$ continuity of pseudo-differential operators. Communications in partial differential equations, 1986, 11(7): 765-778.


\bibitem{Hormander2} L. H\"{o}rmander, Pseudo-differential operators and hypoelliptic equations, Singular integrals (Proc. Sympos. Pure Math., Vol. X, Chicago, Ill., 1966), Amer. Math. Soc., Providence, R.I., 1967, 138-183.


\bibitem{Hormander1} L. H\"{o}rmander, Pseudo-differential operators, Comm. Pure Appl. Math., 18 (1965), 501-517.
\bibitem{HormanderI} L. H\"{o}rmander,  Fourier integral operators.  I,  Acta Math.  1971, 127(2):79-183.


\bibitem{Greenleaf} A. Greenleaf, G. Uhlmann, Estimates for singular Radon transforms and pseudodifferential operators with singular symbols. Journal of functional analysis, 1990, 89(1): 202-232.





\bibitem{Journe} J. Journ\'{e}, Calder\'{o}n-Zygmund operators, pseudo-differential operators and the Cauchy integral of Calder\'{o}n. Springer, 2006.


\bibitem{Nirenberg}  J. Kohn and L. Nirenberg, An algebra of pseudo-differential operators,
Comm. Pure Appl. Math., 18 (1965), 269-305.

\bibitem{Kurtz} D. Kurtz, R. Wheeden, Results on weighted norm inequalities for multipliers, Transactions of the American Mathematical Society, 1979 255: 343-362.



\bibitem{Kenig} C. Kenig, W. Staubach, $\Psi$-pseudodifferential operators and estimates for maximal oscillatory integrals. Studia mathematica, 2007, 183: 249-258.

\bibitem{LW} J. Li, G. Wang, An $ L^{q}\rightarrow L^{r} $ estimate for rough Fourier integral operators and its applications,Discrete and Continuous Dynamical Systems: Series A 2022, 42(11).
\bibitem{Littman}  W. Littman,  $L^p\rightarrow L^q$ estimates for singular integral operators,  Proc.  Symp.  Pure Appl.
Math.  Amer.  Math.  Soc.  1973, 23: 479-481.

\bibitem{Michalowski} N. Michalowski, D. Rule, W. Staubach, Weighted $L^{p}$ boundedness of pseudodifferential operators and applications. Canadian mathematical bulletin, 2012, 55(3): 555-570.

\bibitem{Michalowski1} N. Michalowski, D. Rule, W. Staubach, Weighted norm inequalities for pseudo-pseudodifferential operators defined by amplitudes. Journal of Functional Analysis, 2010, 258(12): 4183-4209.

\bibitem{MiyachiY} A. Miyachi, K. Yabuta, Sharp function estimates for pseudo-differential operators of class $S^{m}_{\varrho,\delta}$, Bulletin of the Faculty of Science, Ibaraki University. Series A, Mathematics, 1987, 19: 15-30.



\bibitem{Miyachi} A. Miyachi, On some singular Fourier multipliers, Journal of the Faculty of Science, the University of Tokyo. Sect. 1 A, Mathematics, 1981, 28(2): 267-315.




\bibitem{Miyachi} A.  Miyachi,
 On some estimates for the wave equation in $L^p$ and $H^p$,  J.  Fac.  Sci.  Tokyo 1980, (27): 331-54.





\bibitem{Park}  B. Park, N. Tomita, Sharp maximal function estimates for linear and multilinear pseudo-differential operators. Journal of Functional Analysis (2024): 110661.


\bibitem{Park1} J. Park, Boundedness of pseudo-differential operators of type $(0,0)$ on Triebel-Lizorkin and Besov spaces. Bulletin of the London Mathematical Society, 2019, 51(6): 1039-1060.

\bibitem{Park3} B. Park, N. Tomita, Sharp Maximal function estimates for Multilinear pseudo-differential operators of type (0,0).  arXiv preprint arXiv:2405.02093 (2024).

\bibitem{Peral}J. Peral,
$L^p$ estimates for the wave equation,  J.  Funct.  Anal.  1980, (36): 114-145.

\bibitem{Ruzhansky} M. Ruzhansky, S. Mitsuru, Weighted Sobolev $L^{2}$ estimates for a class of Fourier integral operators,  Mathematische Nachrichten, 2011, 284(13): 1715-1738.


\bibitem{Stein} E. Stein, Harmonic Analysis: Real-Variable Methods, Orthogonality, and Oscillatory Integrals, volume 43 of Princeton Mathematical Series. Princeton University Press, NJ, 1993.

\bibitem{SSS}  A. Seeger, C. D. Sogge, E. M. Stein,  Regularity properties of Fourier integral operators,  Ann.  Math.  134 (1991), 231-251.




\bibitem{SZ}   S. Shen, X. Zhu, Fourier integral operators on $L^{p}$ when $2<p\leq\infty$.  Analysis and Mathematical Physics, 2023, 13(4): 62.


\bibitem{Tao}T. Tao,  The weak-type $(1, 1)$ of Fourier integral operators of order $-(n-1)/2$.
J.  Aust.  Math.  Soc.  2004, 76(1): 1-21.







\bibitem{CW} G. Wang, W. Chen, A pointwise estimate for pseudo-differential operators. Bulletin of Mathematical Sciences, 2023, 13(02): 2250001.


\bibitem{W} G. Wang, Sharp function and weighted $L^{p}$ estimates for pseudo-differential operators with symbols in general H\"{o}rmander classes. arXi preprint arXi:2206.09825, 2022.
\bibitem{WCY} G. Wang, W. Chen, J. Yang, On the global $L^{\infty}\rightarrow BMO$ mapping property for Fourier integral operators. Analysis and Applications, 2022 20(01): 19-33.

\bibitem{WCY1} J. Yang, G. Wang, W. Chen, On $L^{p}$-boundedness of Fourier integral operators, Potential Analysis, (2022): 1-13.






















\end{thebibliography}
Choose $a$ and $b$ so that $0<-m<a<n$, $0<b<n$, $a+b+m>n$, and let
\begin{center}
$u(x)=|x|^{-a}\chi_{|x|<\eta}(x)$, $\omega=|x-x_{0}|^{-b}$,
\end{center}
where $x_{0}=(1,0,...,0)$ and $\eta>0$ small enough. Clearly, $u\in L^{1}$ and $\omega\in A_{1}$. Now we check $T_{a}u\notin L^{1,\infty}_{\omega}$.
Notice that
$$T_{a}u(x)=G_{-m}(\cdot-x_{0})\ast u(x),$$
where $G_{-m}(x)=\big((1+|\cdot|^{2})^{\frac{m}{2}}\hat{\big)}(x)$ is the Bessel potential of order $-m$. It is well known $G_{-m}(x)\sim |x|^{-n-m}$ near $x=0$. Then we see that
\begin{center}
$T_{a}u(x)\geq C_{0}|x-x_{0}|^{-a-m}$ near $x=x_{0}.$
\end{center}
Therefore, $\forall\epsilon>0$ small enough
\begin{eqnarray*}
&&C_{0}\epsilon^{-a-m}\omega(\{x\in\mathbb{R}^{n}:|T_{a}u(x)|>C_{0}\epsilon^{-a-m}\})\\
&\geq&C_{0}\epsilon^{-a-m}\omega(\{x\in\mathbb{R}^{n}:|T_{a}u(x)|>C_{0}\epsilon^{-a-m},|x-x_{0}|<\epsilon\})\\
&\geq& C_{0}\epsilon^{-a-m}\omega(\{x\in\mathbb{R}^{n}: |x-x_{0}|<\epsilon\})\\
&\geq& C_{0}\epsilon^{-a-m-b+n}
\end{eqnarray*}
The proof is finished by the fact $-a-m-b+n<0$.

\end{sloppypar}
\end{document}